\documentclass[12pt,leqno]{amsart}
\usepackage{graphicx}
\usepackage{epsfig,lpic,wrapfig}
\usepackage{amssymb,hyperref}
\usepackage{color,comment}

\numberwithin{equation}{section}

\newtheorem{thm}{Theorem}[section]

\newtheorem{lem}[thm]{Lemma}
\newtheorem{cor}[thm]{Corollary}
\newtheorem{prop}[thm]{Proposition}

\newtheorem{quest}[thm]{Problem}

\newtheorem{defn}[thm]{Definition}
\theoremstyle{remark}
\newtheorem{rem}[thm]{Remark}
\newtheorem{ex}[thm]{Example}

\newcommand{\tref}[1]{Theorem~\ref{#1}}
\newcommand{\cref}[1]{Korollar~\ref{#1}}

\newcommand{\CAT}{\mathrm{CAT}}

\newcommand{\R}{\mathbb{R}}

\begin{document}
	\pagebreak
	
	
	\title{Weak topology on CAT(0) spaces}

	
	\author{Alexander Lytchak}

\author{Anton Petrunin}	
	

	

	\begin{abstract}
	We analyze weak convergence on $\CAT(0)$ spaces and the existence and properties of corresponding weak topologies.
	\end{abstract}



	\maketitle

\section{Introduction}
Weak convergence and coarse topologies in $\CAT(0)$ spaces have appeared in relation to very different problems and settings in the last years, see \cite{Jost,Monod,Kirk,Bac3,Kell,Streets,Darvas, Gigli} and the survey \cite{Bac} for an overview.
On the other hand, some related fundamental questions have remained open.
This note aims to close some of these gaps. 

\begin{defn}
A bounded sequence $(x_n)$ in a $\CAT(0)$ space $X$ \emph{converges weakly} to a point $x$ if for any compact geodesic $c$ starting at $x$, the closest-point projections $Proj _c (x_n)$ of $x_n$ to $c$ converge to $x$. 
\end{defn}

This notion of convergence (also known as $\Delta$-convergence), introduced in \cite{Jost}, generalizes weak convergence in Hilbert spaces. It can be defined in many other natural ways and is suitable for questions concerning the existence of fixed points and gradient flows, see \cite{Bac}. The weak convergence generalizes verbatim to convergence of nets and satisfies natural compactness and separation properties.

We begin by resolving the question asked by William Kirk and Bancha Panyanak in \cite[Question 1]{Kirk} and discussed, for instance, in \cite{Bac4,Bac,Kell,Obs}.
The question concerns the existence of a weak topology inducing the weak convergence.
Somewhat surprisingly, the answer is different for sequences and for general nets.
In the case of sequences, the answer is always affirmative and the proof is general nonsense, not involving geometry:

\begin{thm} \label{thm1}
Let $X$ be a $\CAT(0)$ space. There exists a unique topology $\mathcal T_{\Delta }$ on $X$ with the following two properties:
\begin{itemize}

\item 	A sequence $(x_n)$ converges in $X$ with respect to $\mathcal T_{\Delta}$ to a point $x$ if and only if the sequence is bounded and converges to $x$ weakly. 

\item The topology $\mathcal T_{\Delta }$ is sequential.

\end{itemize}
\end{thm}

Recall, that a topology $\mathcal T$ is called \emph{sequential} if a subset is $\mathcal T$-closed whenever it contains any $\mathcal T$-limit point of any sequence of its elements. 

This topology $\mathcal T_{\Delta}$, which we want to call the \emph{weak topology}, has the following additional properties;
see Proposition \ref{prop: delta}, Corollary \ref{prop: comp}: 
$\mathcal T _{\Delta}$ is sequentially Hausdorff; any metrically closed, bounded,  convex subset of $X$ is $\mathcal T_{\Delta}$-closed, 
	$\mathcal T_{\Delta}$-sequentially compact and $\mathcal T_{\Delta}$-compact. However,

\begin{prop} \label{prop: haus}
There exists a bounded, separable, two-dimensional $\CAT(0)$ simplicial complex $X$
such that $\mathcal T_{\Delta}$ is not Hausdorff.
\end{prop} 

Together with Proposition \ref{prop: haus}, the next theorem implies that, in general, there is no topology on a $\CAT(0)$ space which induces the weak convergence of \emph{nets}:

\begin{thm} \label{thm2}
Let $X$ be a $\CAT(0)$ space and let $\mathcal T_{\Delta}$ be the weak topology 	defined in Theorem \ref{thm1}. 
For a topology $\mathcal T$ on $X$ the following two conditions are equivalent:

\begin{itemize}

\item A bounded net $(x_{\alpha})$ converges to a point $x\in X$ weakly if and only if $(x_{\alpha})$ converges to $x$ with respect to $\mathcal T$.

\item The restriction of $\mathcal T$ to any closed ball in $X$ is Hausdorff and coincides with $\mathcal T_{\Delta}$. 

\end{itemize}
\end{thm}

We discuss $\mathcal T_{\Delta}$ in some examples and relate this topology to another coarse topology,
the \emph{coconvex} topology introduced by Nicolas Monod in \cite{Monod}. This coconvex topology $\mathcal T_{co}$ on a $\CAT(0)$ space $X$ is defined as the coarsest topology $\mathcal T$ on $X$ for which all metrically closed, convex subsets are $\mathcal T$-closed. 

Every metrically closed, bounded convex subset of $X$ is $\mathcal T_{co}$-compact and $\mathcal T_{co}$-sequentially compact, see Section \ref{sec: conv}.
The weak topology $\mathcal T_{\Delta}$ is finer than the coconvex topology $\mathcal T_{co}$ (Proposition \ref{prop: delta}); these topologies can be different even for bounded $\CAT(0)$ spaces $X$ (Lemma \ref{lem: differ}).
The topologies $\mathcal T_{co}$ and $\mathcal T_{\Delta}$ coincide on all bounded subsets of $X$ if and only if the topology $\mathcal T_{co}$ is sequential and sequentially Hausdorff on bounded convex subsets. 
Whenever $\mathcal T_{co}$ is Hausdorff on bounded subsets, the topologies $\mathcal T_{co}$ and $\mathcal T_{\Delta}$ coincide on bounded subsets. 

Whenever the $\CAT(0)$ space $X$ is locally compact, the metric topology $\mathcal T_{metric}$ coincides with $\mathcal T_{\Delta}$.
On the other hand, for smooth 3-dimensional Riemannian $\CAT(-1)$ manifolds or symmetric spaces of higher rank, the coconvex topology $\mathcal T_{co}$ can be non-Hausdorff and not first countable, as we will observe in Section \ref{sec: conv}.
The failure of the Hausdorff property for symmetric spaces has been expected in \cite{Monod},  a first explicitly  confirmed failure
of the Hausdorff property for some $\CAT(0)$ space seems to be the example of the Euclidean cone over a Hilbert space provided by Martin Kell in \cite{Kell}.

On the other hand,
$\mathcal T_{co}$ is Hausdorff (and therefore coincides with $\mathcal T_{\Delta}$ and induces the weak convergence of nets) in some geometric cases:

\begin{prop} \label{prop: cohaus}
The topology $\mathcal T_{co}$ is Hausdorff in the following cases:
\begin{enumerate}
\item $X$ is homeomorphic to the plane.
\item $X$ is a Riemannian manifold with pinched negative curvature. 
\item $X$ is a finite-dimensional cubical complex.
\end{enumerate}
\end{prop}

The answer we provide to the second point above is a direct consequence of the construction of convex hulls in manifolds with pinched negative curvature due to Michael Anderson. While the main construction of \cite{Anderson} works without changes in infinite dimensions,
it seems not to be sufficient to answer another question from \cite{Monod}:

\begin{quest}
Is the coconvex topology $\mathcal T_{co}$ Hausdorff on the infinite-dimensional complex projective space $X=CH^{\infty}$? 
\end{quest}

Despite Lemma \ref{lem: differ} and the examples by Alano Ancona \cite{Ancona}, see Example \ref{ex: anc} below,
we do not know the answer to the following:

\begin{quest}
Find an example of bounded $\CAT(0)$ spaces for which $\mathcal T_{\Delta}$ is Hausdorff but different from $\mathcal T_{co}$. 
\end{quest}

A natural question is whether for the class of non-locally compact $\CAT(0)$ spaces appearing in most applications, as in \cite{Monod,Streets,Darvas,Clarke,Clarke2}, the weak topology is Hausdorff, at least when restricted to bounded subsets. Most of the examples are subsumed by or related to the example in the following question (we refer to \cite{Monod} for the definition and properties of the spaces of $L^2$-maps):

\begin{quest}
	Let $\Omega$ be a probability space and $X$ a locally compact $\CAT(0)$ space. What are the separation properties of the weak and the coconvex topologies on the space of $L^2$-maps $\mathcal L^2 (\Omega, X)$?
\end{quest}

Also, the following question seems to be very natural in view of the somewhat cumbersome formulation of Theorem \ref{thm2}:

\begin{quest}
	If the restriction of the weak topology $\mathcal T_{\Delta}$ on any bounded subset is Hausdorff, does it have to be a Hausdorff topology on $X$?
\end{quest} 

The paper arose in an attempt to better understand the behavior of convex subsets and convex hulls in $\CAT(0)$ spaces. The non-Hausdorff properties of $\mathcal T_{co}$ should be related to Gromov's question:

\begin{quest}
Is the closed convex hull of a compact subset of any $\CAT(0)$ space compact?
\end{quest}

The paper is structured as follows. In Section \ref{sec: seq} we recall some basic properties of the weak convergence and provide a rather straightforward proof of Theorem \ref{thm1}. In Section \ref{sec: ex} we provide the example verifying Proposition \ref{prop: haus}.
In Section \ref{sec: funct} we prove Theorem \ref{thm2}. The main technical point in the proof is a $\CAT(0)$-version of the theorem of Eberlein--Smulian in functional analysis, saying that a bounded subset is weakly closed if and only if it is weakly sequentially closed (Proposition \ref{prop: eber}).
In the final Section \ref{sec: conv}, we discuss the relations with the coconvex topology. 

\subsection*{Acknowledgments}
We would like to thank Tamas Darvas for explaining to us the problem of the existence of the weak topology on 
$\CAT(0)$ spaces, to Nicolas Monod for his interest and helpful exchange about convex subsets of $\CAT (0)$ spaces and to Miroslav Bačák for helpful comments on a preliminary version of the paper. 
Alexander Lytchak was partially supported by the DFG grant SPP 2026.
Anton Petrunin was partially supported by NSF grant DMS-2005279.

\section{Preliminaries} \label{sec: prel}

\subsection{CAT(0)}
We assume familiarity with the geometry of $\CAT(0)$ spaces and refer to \cite{BH}, \cite{AKP}, \cite{AKP_inv}. 
All $\CAT(0)$ spaces here are by definition complete and geodesic.

By $d(x,y)=d_X(x,y)$ we denote the distance in a metric space $X$.
By $B_r(x)$ we denote the closed ball of radius $r$ around the point $x$.

Any bounded subset $A$ in a $\CAT(0)$ space $X$ has a unique circumcenter $x\in X$ such that 
for some $r=r(A)\in \R$, the \emph{circumradius} of $A$, we have $A\subset B_r(x)$ but $A\not\subset B_r(y)$,
for any other point $y\neq x$ \cite{BH}.

\subsection{General topology}
We refer to \cite{Engel} for a detailed explanation of the notions below.

A \emph{directed set} $I$ is a partially ordered set, such that for any pair $\alpha_1, \alpha_2 \in I$ there exist
$\alpha$ with $\alpha \geq \alpha _1$ and $\alpha \geq \alpha _2$.

A net $(x_{\alpha})$ in a set $X$ is given by a map $\alpha \to x_{\alpha}$ from a directed set $I$ to $X$.
We will mostly suppress the directed set $I$ since it will not play any special role.

In a topological space $X$, a net $(x_{\alpha})$ converges to a point $x$ if for any neighborhood $U$ of $x$
there exists some $\alpha _0$ such that, for all $\alpha \geq \alpha _0$, the elements $x_{\alpha}$ are contained in $U$.

In a topological space, convergence of nets can be used as the right generalization of convergence of sequences from the theory of metric spaces.
For instance, a topological space is Hausdorff if and only if any net can converge to at most one point in $X$.
A topological space is compact if and only if any net in $X$ has a converging subnet.
The closure of a subset $A\subset X$ consists of all limit points $x$ of all nets $(x_{\alpha})$ with $x_{\alpha} \in A$. 

Replacing in the above statement general nets by sequences, we obtain the following properties of spaces which will appear below.

A topological space $X$ is called \emph{sequentially Hausdorff} if
any sequence in $X$ has at most one limit point. Any Hausdorff space is sequentially Hausdorff but the opposite does not hold.

A topological space $X$ is called \emph{sequentially compact} if any sequence in $X$ has a convergent subsequence. A compact space does not need to be sequentially compact and a sequentially compact space does not need to be compact.

\subsection{Basics on weak convergence}
Let $X$ be a $\CAT(0)$ space.
We stick to the definition of weak convergence given in the introduction and refer to \cite{Bac} for other descriptions and for the explanations and references of the following properties frequently used below:

Any bounded net in $X$ has at most one weak limit point in $X$.
Any subnet of a weakly converging net converges weakly to the same limit point. 
Any bounded sequence (net) has a weakly converging subsequence (subnet).

\section{Weak convergence of sequences} \label{sec: seq}

In this section we provide the rather straightforward:
\begin{proof}[Proof of Theorem \ref{thm1}]
Define the topology $\mathcal T_{\Delta}$ as follows. We say that a subset $A\subset X$ is $\mathcal T_{\Delta}$-closed, if, for any bounded sequence $x_n \in A$ weakly converging to a point $x\in X$,
we have $x\in A$. 

By definition, the empty set and the whole set are $\mathcal T_{\Delta}$-closed. Moreover,
any intersection of $\mathcal T_{\Delta}$-closed subsets $A_{\alpha}$ is 
$\mathcal T_{\Delta}$-closed. 

Finally, if $A_1,...,A_m$ are $\mathcal T_{\Delta}$-closed and $(x_n)$ is a bounded sequence in $A_1\cup.... \cup A_m$ weakly converging to $x$, then we find a subsequence of $(x_n)$ contained in one of the $A_i$. This subsequence also weakly converges to $x$, therefore $x\in A_i$. Hence $A_1\cup...\cup A_m$ is $\mathcal T_{\Delta}$-closed.

Altogether, this shows that the family of all $\mathcal T_{\Delta}$-closed sets is the family of closed sets of a topology, which we will denote by $\mathcal T_{\Delta}$.

We claim that a sequence $(x_n)$ in $X$ converges to a point $x$ with respect to $\mathcal T_{\Delta}$ if and only if $(x_n)$ is bounded and converges to $x$ weakly.

Firstly, let $(x_n)$ be bounded and weakly converge to $x$. If $(x_n)$ does not $\mathcal T _{\Delta}$-converge to $x$, we would find a $\mathcal T_{\Delta}$-open subset $U$ containing $x$ and a subsequence $(x_{m_n})$ contained in the complement $A\mathrel{:=}X\setminus U$.
However, $(x_{m_n})$
also converges to $x$ weakly, hence, by the definition of $\mathcal T_{\Delta}$-closed subsets, we infer $x\in A$, a contradiction.

On the other hand, let a sequence $(x_n)$ converge in the $\mathcal T_{\Delta}$-topology to $x$. If $(x_n)$ is not bounded, we could find a subsequence $(x_{m_n})$ such that $d(x_1,x_{m_n}) \to \infty$. Then the countable set $\{x_{m_n} \}$ is $\mathcal T_{\Delta}$-closed. Hence, $(x_{m_n})$ does not $\mathcal T_{\Delta}$-converge to $x$. Therefore, $(x_n)$ must be bounded.

Assume that $x_n$ does not converge weakly to $x$. Then we find a subsequence
$(x_{m_n})$ of $(x_n)$ which converges weakly to some point $y\neq x$. Moreover, deleting finitely many elements from the sequence, we may assume that $x_{m_n}$ is not equal to $x$ for all $n$.
Then the union $A$ of all
$x_{m_n}$ and the point $y$ is $\mathcal T_{\Delta}$-closed. Thus, the complement of $A$ is a $\mathcal T_{\Delta}$-open neighborhood of $x$, which does not contain all but finitely many elements of the sequence $(x_n)$. This contradiction proves that $(x_n)$ weakly converges to $x$ and finishes the proof of the claim. 

The claim and the definition of $\mathcal T_{\Delta}$ imply that a subset $A$ of $X$ is $\mathcal T_{\Delta}$-closed if every $\mathcal T_{\Delta}$-limit $x\in X$ of a sequence of points in $A$ is contained in $A$. This means that $\mathcal T_{\Delta}$ is sequential.

We have verified the required properties of $\mathcal T_{\Delta}$. Let $\mathcal T$
be another sequential topology on $X$, for which a sequence $(x_n)$ converges 
to $x$ if and only if $(x_n)$ is bounded and weakly converges to $x$. Then, for $\mathcal T$ and $\mathcal T_{\Delta}$ the convergence of sequences coincide. Since both topologies are sequential, this implies that the properties of being closed with respect to $\mathcal T$ and $\mathcal T_{\Delta}$ coincide. Hence, $\mathcal T=\mathcal T_{\Delta}$.
\end{proof}

Basic properties of the weak topology $\mathcal T_{\Delta}$ are direct consequence
of the definition and the corresponding properties of weak convergence:

\begin{prop} \label{prop: delta}
The weak topology $\mathcal T_{\Delta}$ on a $\CAT(0)$ space $X$ 
is finer than the coconvex topology and coarser than the metric topology:
$$\mathcal T_{co} \subset \mathcal T_{\Delta}\subset \mathcal T_{metric} .$$
The topology $\mathcal T_{\Delta}$ is sequentially Hausdorff. Any metrically closed, bounded, convex subset $C\subset X$ 
is $\mathcal T_{\Delta}$-sequentially compact.
\end{prop}

The less trivial statement that any closed, bounded, convex subset is $\mathcal T_{\Delta}$-compact will be derived later in Corollary \ref{prop: comp}. 

We finish the section with two simple examples. The first example is a direct consequence of the definition and the theorem of Hopf--Rinow:

\begin{ex}
Assume that the $\CAT(0)$ space $X$ is locally compact. Then $\mathcal T_{\Delta}$ coincides with the metric topology. 
\end{ex}

The second example is a special case of the fact that 
the weak convergence as defined above corresponds to the usual weak convergence in the case of Hilbert spaces, \cite{Bac} and Theorem of Eberlien-Smulian, \cite{Eberlein}, in the case of Hilbert spaces, saying that a subset is compact in the weak topology if and only if it is sequentially compact. 

\begin{ex}
For a Hilbert space $X$, the topology $\mathcal T_{\Delta}$ coincides with the weak topology of the Hilbert space and with $\mathcal T_{co}$. 
\end{ex}

\section{Example} \label{sec: ex}
We are going to show that $\mathcal T_{\Delta}$ can be non-Hausdorff:

\medskip

\noindent\emph{Proof of Proposition \ref{prop: haus}.}
Let $Y_1$ be a countable family of intervals $[0, \frac \pi 4]$ glued together at the common boundary point $0$.
Fix an endpoint $b$ among the countably many endpoints of the tree $Y_1$.
Choose a countable family of isometric copies of $Y_1$ and glue all of them together by identifying the chosen "endpoints" $b$ with each other.

The arising space $Y$ is a tree with a special point $p$ (the point at which all subtrees isometric to $Y_1$ are glued together).
Point $p$ is the unique circumcenter of the simplicial tree $Y$.
The tree has countably many branches at $p$ and every point at distance $\frac \pi 4$ from $p$.
There are no other branching points in $Y$; all edges of the tree $Y$ have length $\frac \pi 4$. 

\begin{wrapfigure}{o}{41 mm}
\centering
\includegraphics{mppics/pic-1}
\end{wrapfigure}

We denote by $E$ the set of endpoints of the tree $Y$ and by $B$ the set of the branching points
different from $p$ (thus the $\frac \pi 4$-sphere around $p$).
Any pair of different points of $B$ lie at distance $\frac \pi 2$ from each other. Any pair 
of different points in $E$ either are at distance $\pi$ and have $p$ as the midpoint or are at distance $\frac \pi 2$ and have a point from $B$ as their midpoint. 

Let $\hat X$ denote the Euclidean cone over $Y$. We identify $Y$ with the unit sphere around the tip $o$ in $\hat X$. For a point $y\in Y$ and a number $\lambda \geq 0$, we denote by $\lambda \cdot y$ the point in the cone $\hat X$ on the radial ray in the direction of $y$ at distance $\lambda$ from the vertex $o$. 

For any edge $I$ of $Y$ with endpoints $y_1,y_2$ consider the triangle $S_I$ defined by the points $o, 2\cdot y_1, 2\cdot y_2$ in $\hat X$. The union of all such triangles is a closed convex subset $X$ of $\hat X$. This subset $X$ is bounded and contains the unit ball $B_1(o)$. Moreover, $X$ is a $2$-dimensional simplicial complex with countably many simplices.

We are going to verify that the points $o$ and $\frac 1 2 \cdot p$ are not separated in 
the weak topology $\mathcal T_{\Delta}$ on space $X$.

Firstly, for any pair of different points in $E\subset Y\subset X$ the unique geodesic in $X$ connecting them 
either has its midpoint in $o$ (if the points are at distance $\pi$ in $Y$) or it has its midpoints in $\frac 1 {\sqrt 2} \cdot b$ for the unique midpoint $b\in B$ of the corresponding geodesic in $Y$. 

Given any sequence $(x_n)$ of elements in $E\subset Y\subset X$ with pairwise distance $\pi$ in $Y$, we see that the convex hull of $\{ x_n \}$ is the union of the geodesic segments $[o,x_n]$,
thus a tree with a unique vertex in $o$.
In this case, the sequence $(x_n)$ converges weakly to $o$.

Given any sequence $x_n$ of pairwise different elements in $E$ with pairwise distance $\frac \pi 2$ in $Y$, the convex hull of $\{ x_n \}$ is again a tree with a unique vertex
$\frac 1 {\sqrt 2} \cdot b$, the common
midpoint of any pair of different points in the sequence $(x_n)$. Thus, $(x_n)$ weakly converges to $\frac 1 {\sqrt 2} \cdot b$.

Similarly, for any sequence of different point $b_n \in B\subset Y\subset X$, the sequence $(b_n)$ weakly converges in $X$ to the point $\frac 1 {\sqrt 2} \cdot p$. Thus, by rescaling, the sequence $\frac 1 {\sqrt 2} \cdot b_n$ converges weakly to $\frac 1 2 \cdot p$.

Assume that $o$ and $\frac 1 2 \cdot p$ can be separated in $\mathcal T_{\Delta}$. Thus, we find $\mathcal T_{\Delta}$-closed subsets $C_1$ and $C_2$ such that $o\notin C_1$, $\frac 1 2 \cdot p\notin C_2$ and $C_1\cup C_2 = X$.

By above, $C_1$ cannot contain infinitely many points of $E$, which have in $Y$ pairwise distance $\pi$. 

Thus, for all but finitely many branch-points $b\in B \subset Y$ all points in $E$ at distance $\frac \pi 4$ from $b$ are contained in $C_2$. By above, for any such $b$ we must have $\frac 1 {\sqrt 2} \cdot b \in B$. Since we have infinitely many such points, we conclude $\frac 1 2 \cdot p\in C_2$, in contradiction to our assumption.

Thus, we have verified that $(X, \mathcal T_{\Delta})$ is not Hausdorff.
\qed

The provided example implies that $\mathcal T_{co}$ and $\mathcal T_{\Delta}$ may be different: 

\begin{lem} \label{lem: differ}
	The weak topology $\mathcal T_{\Delta}$ and coconvex topology $\mathcal T_{co}$ do not coincide on the bounded $\CAT(0)$ space $X$ constructed above.
\end{lem}

\begin{proof}
Consider the set 
$$A=E\cup \tfrac 1 {\sqrt 2} \cdot B \cup \{\tfrac 1 2 \cdot p\} \cup \{o\}$$ which has appeared above.
As explained in the proof of Proposition \ref{prop: haus} above, the set $A$ is $\mathcal T_{\Delta}$-closed.

We are going to prove that $\frac 1 4 \cdot p$ is contained in the $\mathcal T_{co}$ closure of $A$. Assuming the contrary, we find finitely many convex, metrically closed subsets $C_1,...,C_n$ in $X$ which cover $A$ and do not contain $\frac 1 4 \cdot p$. 

For any $b\in B$, consider the set $E^b$ of points in $E$ which are at distance $\frac \pi 4$ from $b$. Then a counting argument implies that at least one of the sets $C_i$ contains
at least 2 points in any of the sets $E^{b_1},E^{b_2}$, for different $b_1,b_2\in B$.
 Then this convex set $C_i$ contains the origin $o$ (as the midpoint of a point in $E^{b_1}$ and $E^{b_2}$), the points $\frac 1 {\sqrt 2} \cdot b_{i}$ and therefore their midpoint $\frac 1 2 \cdot p$. Hence, $C_i$ also contains the whole geodesic $[o,\frac 1 2 \cdot p]$ and, therefore, 
 $\frac 1 4 \cdot p \in C_i$, in contradiction to our assumption.

Thus, the set $A$ is not $\mathcal T_{co}$-closed, finishing the proof.
	\end{proof}

\section{Compactness} \label{sec: funct}

The following result can be seen as an analog of the theorem of Eberlein--Smulian in functional analysis. Unlike Theorem \ref{thm1}, here the $\CAT(0)$ geometry plays an important role several times: 

\begin{prop} \label{prop: eber}
	Let $(x_{\alpha})$ be a bounded net in a $\CAT(0)$ space $X$ weakly converging to a point $x$.
Then there exists a sequence $x_{\alpha_1},x_{\alpha_2},\dots$ of elements of the net weakly converging to $x$.
\end{prop}

\begin{proof}
	Replacing the net by a subnet we may assume that the net $r_{\alpha}\mathrel{:=}d(x_{\alpha},x)$ 
	of real numbers converges to some $r\geq 0$. If $r=0$, we find some $x_{\alpha _i}$ such that $\lim _{i\to \infty} r_{\alpha_i} =0$. Thus, the sequence $x_{\alpha _i}$ converges to $x$ in the metric topology, and, therefore, also weakly. 
Thus, we may assume $r>0$ and, after rescaling, $r=1$. 

We choose inductively $\alpha _k \in I$, for $k=1,2,...$, starting with an arbitrary $\alpha _1$. Let the elements $\alpha _1\leq ...\leq \alpha _k$ in $I$ be already chosen.	 

For any non-empty subset $S \subset \{1,...,k\}$, denote by $m_S$ the unique circumcenter $m_S$ of the finite set $\{x_{\alpha _i}, i\in S\}$.
Since the net $(x_{\alpha})$ converges weakly to $x$ and $(r_{\alpha})$ converges to $1$, we find some $\alpha _{k+1} \geq \alpha _k$ with the two following properties, for any $\alpha \geq \alpha _{k+1}$: 

1) $|r_{\alpha} -1| \leq 2^{-k-1}$.

2) For all nonempty $S\subset \{1,...,k\}$ the projection $Proj _c (x_{\alpha})$ of $x_{\alpha}$ onto the geodesic
$c=[xm_S]$ has distance at most $2^{-k-1}$ from $x$.

Note that any subsequence of the sequence $(x_{\alpha_i})$ has also the properties (1) and (2). 
We claim that the so-defined sequence $(x_{\alpha _i})$ converges to $x$ weakly. 
The proof of the claim relies only on the strict convexity of the squared distance functions
and is rather straightforward.
For the convenience of the reader, we present the somewhat lengthy details. 

Assuming the contrary and replacing the sequence by a subsequence we may assume that the sequence converges weakly to a point $z\neq x$.
Set $\delta\mathrel{:=} d(z,x)$.
Choosing yet another subsequence we may assume that $s_{\alpha_i}\mathrel{:=} d(x_{\alpha_i} , z)$ converge to some $s \geq 0$, for $i\to \infty$. 

We set $\varepsilon \leq \frac {\delta ^2} {10}$ and find some $i_0$ such that 
$(1- 2^{-i_0-1})^2 > 1 -\varepsilon$ and such that, for all $i \geq i_0$,
$$|r_{\alpha_i} ^2 -1| <\varepsilon \; ; \; |s_{\alpha_i} ^2-s^2| <\varepsilon . $$
Using the weak convergence of $(x_{\alpha_i})$ to $z$ and $\CAT(0)$ comparison, we may assume in addition, that for all $i\geq i_0$ 
$$d^2(x_{\alpha _i}, x) -d^2(x_{\alpha_i} ,z) ^2 \geq	\delta ^2 -\varepsilon = 9 \varepsilon .$$

For $j=1,2...$ we consider the point $p_j\mathrel{:=} x_{\alpha _{i_0+ j}}$. By above, the circumradius $t$
of the countable set $\{p_j\}$ satisfies 
$$t^2 < 1 -\tfrac 1 2 \delta ^2 = 1-5\varepsilon \,.$$ 
Denote by $0\leq t_k \leq t$ the circumradius of the set $\{p_1,...,p_k\}$.
We claim that there exists some positive $\rho >0$, such that $t^2_{k+1}-t^2_k >\rho$ for all $k\geq 1$. Since the sequence $(t_k)$ is bounded above by $t$, this would provide a contradiction and finish the proof.

In order to prove the claim, consider the circumcenter $m_k$ of the subset $p_1,...,p_k$.
Thus, $m_k$ is the point at which the $2$-convex function, 
$$f(y)\mathrel{:=} \max_{1\leq i \leq k} d^2 (y,p_i)$$ 
assumes its unique minimum $t_k ^2$. By the $2$-convexity, we deduce
$$f(m_{k+1}) \geq t_k ^2 + d^2 (m_k,m_{k+1}).$$
On the other hand, $f(m_{k+1}) \leq t^2_{k+1}$, hence 
$$t_{k+1} ^2 \geq t_k ^2 + d^2 (m_k,m_{k+1}).$$ 

By construction of the sequence $x_{\alpha_i}$, we have 
$$d^2(p_{k+1}, m_k) \geq (1- 2^{-i_0-1})^2 > 1 -\varepsilon .$$
Thus, by the triangle inequality and the fact 
$$d^2 (p_{k+1} , m_{k+1}) \leq t_{k+1} ^2 \leq t^2 \leq 1-5\varepsilon \,$$
we obtain some positive lower bound $\rho>0$ on $d^2(m_k,m_{k+1})$. 
This finishes the proof of the claim and of the proposition. 
\end{proof}

As a consequence, we derive: 

\begin{lem} \label{lem: h}
If a bounded net $(x_{\alpha})$ in $X$ converges to the point $x$ weakly then $(x_{\alpha})$ converges to $x$ with respect to the $\mathcal T_{\Delta}$-topology.
\end{lem} 

\begin{proof}
Assume the contrary.
Then, replacing the net by a subnet, we find a $\mathcal T_{\Delta}$-open neighborhood $U$ of $x$ which does not contain any $x_{\alpha}$. Using 
Proposition \ref{prop: eber} we find a sequence $x_{\alpha_1},....,x_{\alpha _k},...$ of elements of the net converging weakly to $x$. Then $x$ is contained in the $\mathcal T_{\Delta}$-closed set $X\setminus U$ which contains all elements of the sequence. This contradicts the definition of $\mathcal T_{\Delta}$-closed sets. 	
\end{proof}

Since any bounded net has weakly convergent subnets, we infer:

\begin{cor} \label{prop: comp}
	Every bounded $\mathcal T_{\Delta}$-closed subset $A$ of $X$ is $\mathcal T_{\Delta}$-compact. 	
\end{cor}

Now we provide:
\begin{proof}[Proof of Theorem \ref{thm2}]
	Let $\mathcal T$ be a topology on $X$, such that a bounded net $(x_{\alpha})$ weakly converges to $x$ if and only if this net $\mathcal T$-converges to $x$. Since any net has at most one weak limit point and since the Hausdorff property can be recognized by the uniqueness of limit points of nets, we deduce that any bounded subset of $X$ is Hausdorff with respect to $\mathcal T$. 
	
	 Let $A$ be a bounded subset of $X$. By definition, $A$ is $\mathcal T$-closed if and only if it contains all weak limit points of any net $(x_{\alpha})$ of elements in $A$. From 
	 Proposition \ref{prop: eber}, this happens if and only if $A$ contains all weak limit points of any \emph{sequence} of elements in $A$. Thus, if and only if $A$ is $\mathcal T_{\Delta}$-closed. We infer, that $\mathcal T$ coincides with $\mathcal T_{\Delta}$ on bounded subsets.

	Assume, on the other hand, that the weak topology $\mathcal T_{\Delta}$ is Hausdorff on any ball in $X$. We claim that a bounded net $(x_{\alpha})$ converges weakly to $x$ if and only if $(x_{\alpha})$ converges to $x$ with respect to $\mathcal T_{\Delta}$. 
	
Due to Lemma \ref{lem: h}, the only if conclusion always holds.
On the other hand, assume that
$(x_{\alpha})$ converges to $x$ with respect to $\mathcal T_{\Delta}$ but does not weakly converge to $x$.
Replacing $(x_{\alpha})$ by a subnet we may assume that $(x_{\alpha})$ weakly converges to another point $y$. Due to Lemma \ref{lem: h}, this implies that the net $(x_{\alpha})$ converges to the point $y$ with respect to the topology $\mathcal T_{\Delta}$. But this contradicts the assumption that $\mathcal T_{\Delta}$ is Hausdorff on the bounded ball which contains the net $(x_{\alpha})$. 
	
This proves the "if"-direction and finishes the proof of the theorem.	
\end{proof}

\begin{rem}
Using the considerations above, it is not difficult to prove another form of Theorem \ref{thm2}.
Namely, the topology $\mathcal T_{\Delta}$ is Hausdorff on any bounded subset of $X$
(and thus weak convergence of bounded nets is equivalent to the $\mathcal T_{\Delta}$-convergence) if and only if the topology $\mathcal T_{\Delta}$ is \emph{Frechet--Urysohn} on any bounded set. Recall, that a topology $\mathcal T$ is called Frechet--Urysohn, if the closure of any set $A$ in this topology is the set of all $\mathcal T$-limit points in $X$ of all \emph{sequences} contained in $A$. 
\end{rem}

\section{Coconvex topology} \label{sec: conv}
The coconvex topology $\mathcal T_{co}$ is coarser than $\mathcal T_{\Delta}$, Proposition \ref{prop: delta}.
Thus, convergence of sequences (nets) with respect to $\mathcal T_{\Delta}$ implies convergence with respect to $\mathcal T_{co}$.
This immediately implies that any bounded, $\mathcal T_{co}$-closed set is $\mathcal T_{co}$-compact and $\mathcal T_{co}$-sequentially compact.

\begin{prop}
The topologies $\mathcal T_{co}$ and $\mathcal T_{\Delta}$ coincide on all bounded subsets of a $\CAT(0)$ space $X$ if and only if the topology $\mathcal T_{co}$ is sequential and sequentially Hausdorff on every closed ball $B_r(x)$ in $X$.
This happens if $B_r(x)$ is $\mathcal T_{co}$-Hausdorff.
\end{prop}

\begin{proof}
We may replace $X$ by a ball $B_r(x)$ and assume that $X$ is bounded. The only if statement follows from Proposition \ref{prop: delta}. 

On the other hand, assume that $\mathcal T_{co}$ is sequential and sequentially Hausdorff on the bounded $\CAT(0)$ space $X$. Due to Proposition \ref{prop: delta}, the identity map
$Id\colon (X, \mathcal T_{\Delta}) \to (X, \mathcal T_{co} )$ is continuous. 
In order to prove that the inverse $Id\colon(X, \mathcal T_{co}) \to (X, \mathcal T_{\Delta})$ is continuous, consider a $\mathcal T_{\Delta}$-closed subset $A$ and assume that 
$A$ is not $\mathcal T_{co}$-closed. Since $\mathcal T_{co}$ is sequential, we find a sequence $(x_n)$ in $A$ which $\mathcal T_{co}$-converges to a point $x \in X\setminus A$. Using that $X$ is $\mathcal T_{\Delta}$-sequentially compact, we may replace $(x_n)$ by a subsequence and assume that $(x_n)$ converges to some point $y$ in $X$ with respect to $\mathcal T_{\Delta}$.
Then, by Proposition \ref{prop: delta}, the sequence converges to $y$ also with respect to $\mathcal T_{co}$. The assumption that $\mathcal T_{co}$ is sequentially Hausdorff gives us 
$x=y$. Since $A$ is $\mathcal T_{\Delta}$-closed, we deduce $x=y\in A$. This contradiction implies that $Id\colon(X, \mathcal T_{co} ) \to (X, \mathcal T_{\Delta})$ is continuous, hence
$\mathcal T_{\Delta} =\mathcal T_{co}$.

Finally, if $\mathcal T_{co}$ is Hausdorff on the bounded $\CAT(0)$ space $X$, then the compactness
of $\mathcal T_{\Delta}$ and the continuity of the identity map
$Id\colon (X, \mathcal T_{\Delta}) \to (X, \mathcal T_{co} )$ imply $\mathcal T_{\Delta} =\mathcal T_{co}$.
\end{proof}

In the proof of Proposition \ref{prop: cohaus} below, we assume more knowledge of non-positive curvature than in the rest of this paper.
We refer 
to \cite{LN1} for properties of geodesically complete $\CAT(0)$ spaces,
to \cite{Schwer} for properties of cubical complexes,
and to \cite{Anderson} and \cite{Borbely} for manifolds of pinched negative curvature.

\begin{proof}[Proof of Proposition \ref{prop: cohaus}]
Assume first that $X$ is homeomorphic to the plane $\R^2$.
Then each geodesic $\gamma \colon[a,b]\to X$ extends to a geodesic $\hat \gamma \colon\R \to X$, \cite{BH}.
Moreover, by Jordan's theorem, $\hat \gamma$ divides $X$ into two connected components both having $\hat \gamma$ as their boundaries. The closures of the connected components are convex.
Thus, the open components are $\mathcal T_{co}$ open.
	
In order to prove that $\mathcal T_{co}$ is Hausdorff it suffices to find, for any pair of points $x,y$, some geodesic $\gamma \colon\R \to X$, such that $x$ and $y$ are in different components of $X\setminus \gamma$.
We connect $x$ and $y$ by a geodesic $\eta$ and take the midpoint $m$ of $\eta$.
We find two points $p^{\pm}$ sufficiently close to $m$ which lie in different components of $X\setminus \hat \eta$ for some extension of $\eta$ to an infinite geodesic.
Then consider a geodesic $\gamma\colon\R \to X$ which contains $p^+$ and $p^-$.
The geodesic $\gamma$ intersect $\eta$ between $x$ and $y$.
We infer that $x$ and $y$ lie in different components of $X\setminus \gamma$.
This finishes the proof if $X$ is homeomorphic to a plane.

	Assume now that $X$ is a finite-dimensional cubical complex and choose $x,y\in X$. Taking a sufficiently fine cubical subdivision, we may assume that the diameter of all cubes is much smaller than the distance between $x$ and $y$. Then the geodesic between $x$ and $y$ intersects at least one \emph{hyperwall} in $X$. Any such hyperwall is a convex subsets dividing $X$ into two connected and convex components. As above, we deduce that $x$ and $y$ are separated in $\mathcal T_{co}$.

	Finally, let $X$ be a Riemannian manifold with pinched negative curvature and let $x,y\in X$
	be arbitrary different points. Fix $r>d(x,y)$ and set $B=B_r (x)$. By Anderson's construction, \cite{Anderson}, see also \cite[Theorem 2.1]{Borbely}, we find finitely many closed convex subsets $C_i$ in $X\setminus B$, such that $V\mathrel{:=}X\setminus \cup _{i=1}^m C_i$
	is bounded. Then $V$ is a $\mathcal T_{co}$-open set containing $B$ and contained in some larger closed ball $B'$. 
	
	On the compact ball $B'$ the topology $\mathcal T_{co}$ coincides with the metric topology, \cite[Lemma 17]{Monod}, hence it is Hausdorff. Thus, we find $\mathcal T_{co}$-open subsets 
	$U_1$ and $U_2$ in $X$ containing $x$ and $y$, respectively, such that $U_1\cap U_2\cap B'$ is empty. Then $U_1\cap U$ and $U_2\cap U$ are disjoint $\mathcal T_{co}$-neighborhoods of $x$ and $y$ in $X$.
\end{proof}

\begin{rem}
	As the proof and the reference to \cite{Borbely} shows, in condition (2) one can replace the pinching by the assumption that the quotient of the minimal and maximal curvature in the ball $B_r(x_0)$ around some chosen point $x_0$ is at most $2^{\lambda r}$ for some $\lambda \in \R$. Moreover, a closer look at the proof shows that under the assumptions (1) or (2),
	the topology $\mathcal T_{co}$ coincides with the metric topology on all of $X$.	 
\end{rem}

We discuss finally two examples showing that the coconvex topology can be quite strange even for rather regular spaces. Below we denote for a locally compact $\CAT(0)$ space $X$ by $X^{\infty}$ its boundary at infinity with its cone topology, \cite{BH}. Recall that $X^{\infty}$ is compact.

\begin{lem} \label{lem: anc}
	Let $X$ be a locally compact $\CAT(0)$ space. Assume that $X$ is not bounded and for any closed convex subset $A$ of $X$ different from $X$, the boundary at infinity $A^{\infty}$ is nowhere dense in $X^{\infty}$. Then the coconvex topology $\mathcal T_{co}$ on $X$ is non-Hausdorff and not first-countable.
\end{lem}

\begin{proof}
Since $X^{\infty}$ is a compact space, it is not a countable union of nowhere dense subsets, by Baire's theorem. 

Therefore, by our assumption, $X$ is not a finite union of closed convex subsets different from $X$. Thus, any finite intersection of non-empty $\mathcal T_{co}$-open subsets is non-empty.
In particular, $\mathcal T_{co}$ is not Hausdorff.

Assume now that $\mathcal T_{co}$ is first-countable on $X$, fix an arbitrary $x\in X$ and a $\mathcal T_{co}$-fundamental system of its open neighborhoods $U_1,...,U_n,...$.
By definition of $\mathcal T_{co}$, we may assume 
that each $U_i$ is the complement of a finite union of closed convex subsets $K_i ^j$, not containing the point $x$. Hence, the union of the boundaries at infinity 
$$\cup _{i,j} (K_i ^j ) ^{\infty } \subset X^{\infty}$$
is not all of $X^{\infty}$. Consider an arbitrary point $z\in X^{\infty}$ not contained in this union and a ray $\gamma$ in $X$ with endpoint $z \in X^{\infty}$, such that $x$ is not on $\gamma$. Then $X\setminus \gamma$ is a $\mathcal T_{co}$-open neighborhood of $x$ which does not contain any of the set $U_i$. This contradiction shows that $\mathcal T_{co}$ is not first-countable. 
\end{proof}

The first example directly follows from
Lemma \ref{lem: anc} above and \cite[Theorem B, Corollary C]{Ancona}:

\begin{ex} \label{ex: anc}
There exists a smooth $3$-dimensional $\CAT(-1)$ Riemannian manifold $X$ for which the coconvex topology $\mathcal T_{co}$ is not Hausdorff and not first-countable.
\end{ex}

In the final example, we  use some facts about geometry of spherical buildings arising as the boundary at infinity of symmetric space with their corresponding \emph{Tits-metric}, see  \cite{KleinerLeeb}, \cite{KleinerLeeb1}, \cite{Kap}.  The following result might be known to specialists, accordingly to Nicolas Monod it was known to Bruce Kleiner many years ago. 

\begin{prop}
Let $X$ be an irreducible, non-positively curved symmetric space of rank at least two. Then $X$ satisfies the assumptions, and, therefore, the conclusions of Lemma \ref{lem: anc}. 
\end{prop}

\begin{proof}
Assume the contrary and consider any closed convex subset $A$ of $X$ such that the boundary at infinity $A^{\infty}$  of $A$ has non-empty interior in the $(n-1)$-dimensional sphere 
$X^{\infty}$; here $n$ is the dimension of $X$.

Thus, in the cone topology, $A^{\infty}$ has dimension $n-1$. Therefore, there are no totally geodesic symmetric spaces $Y \subsetneq X$ 
with $A^{\infty} \subset Y^{\infty}$.  On the other hand, if $A^{\infty} =X^{\infty}$ then $A=X$. 
Thus, we may assume $A^{\infty} \neq X^{\infty}$.  Applying \cite[Theorem 3.1]{KleinerLeeb}, we deduce that $A^{\infty}$ is not a \emph{sub-building} of the \emph{spherical building} $X^{\infty}$.

Since $A^{\infty}$  contains an open subset in the cone topology, we find a non-empty subset $O$ of $A^{\infty}$, open in the cone topology and consisting of \emph{regular} points only.  If, for some 
$p\in O$, we find an antipode $q\in A^{\infty}$ (with respect to the \emph{Tits-distance}) then $A^{\infty}$  contains a spherical apartment  (the boundary of a maximal flat in $X$), as the convex hull in the Tits-metric of $q$ and a Tits-ball around $p$. By \cite[Theorem 1.1]{BL}, this would imply that $A^{\infty}$ is a sub-building,
 in contradiction to the statements above.   
Thus,  for no  $p\in O$ and $q\in A^{\infty}$ the Tits-distance between 
$p$ and $q$ equals $\pi$.

  We are going to construct a pair of antipodes $p\in O$ and $q\in A^{\infty}$ and achieve a contradiction.  We start with an arbitrary point $p\in O$.

Let $G$ be the isometry group of $X$ (and of  $X^{\infty}$) and denote 
 by $\Delta$ the spherical Coxeter chamber $X^{\infty} /G$  of the spherical building $X^{\infty}$.
 Let $\mathcal P:X^{\infty} \to \Delta$ be the canonical projection.  Denote by $I:\Delta \to \Delta$ the isometry of the Coxeter chamber induced by the action of $-Id$ on any apartment of $X^{\infty}$.   The map $I$ is an involution, which is the identity map if and only if the Coxeter group $W$ of $X^{\infty}$  has a non-trivial center (note, that this is the case for all Weyl groups, which are not of type $A_m$, $E_6$ or $D_{2m+1}$, see \cite[p.71]{Hum}).

Consider the orbit $L:=G\cdot p= \mathcal P^{-1} (\mathcal P (p)) \subset X^{\infty}$. Any element $p'\in L$ is contained in a unique Coxeter chamber $\Delta _{p'}$.  Consider the set $L^{op} _p$   of all elements $p'$ in $L$  which are 
in an \emph{opposite} Coxeter to $p$, thus such that the Coxeter chamber  $\Delta _{p'}$ through $p'$ contains an antipode of $p$.  Then $L^{op} _p$ is open and dense in the manifold $G\cdot p$, see, for instance, \cite{Kap}.   Thus, we find an element  $p'\in O\cap L^{op} _p$.

If the isometry $I:\Delta \to \Delta$ is the identity (see the discussion above), then $p'$ is an antipode of $p$ and we are done.  If $I$ is not the identity, then looking at an apartment through $p$ and $p'$ we deduce that the Tits-geodesic between $p$ and $p'$ contains a point $q$ which is projected by $\mathcal P$ onto $I(p)$.  Then  $L$ contains all antipodes of $q$. By convexity, $q\in A^{\infty}$. As above, the set 
 $L^{op} _q \cap O$ of elements in $O$  contained in a chamber opposite to $q$ is not empty.
 For any such element $p'\in L^{op} _q\cap O$, the distance between $q$ and $p'$ is $\pi$.
 
 Thus, in both cases we have found a pair of antipodes $p\in O$ and $q\in A^{\infty}$, finishing the proof. 
\end{proof}


	

\bibliographystyle{alpha}
\bibliography{Projection}

\end{document}